\documentclass[10pt]{amsart}
\usepackage{amssymb}
\usepackage{bm}
\usepackage{graphicx}
\usepackage[centertags]{amsmath}
\usepackage{amsfonts}
\usepackage{amsthm}
\usepackage{amsbsy}
\usepackage{mathtools}
\usepackage{mathrsfs}
\usepackage{cases}
\usepackage{xfrac}
\usepackage[all]{xy}
\usepackage[linktocpage=true]{hyperref}
\usepackage{hyperref}
\hypersetup{
    colorlinks=true,
    linkcolor=blue,
    filecolor=blue,
    citecolor=blue,
    urlcolor=cyan}
\usepackage[margin=4cm]{geometry}
\newtheorem{thm}{Theorem}[section]
\newtheorem{cor}[thm]{Corollary}
\newtheorem{lem}[thm]{Lemma}
\newtheorem{sbl}[thm]{Sublemma}
\newtheorem{prop}[thm]{Proposition}

\newtheorem{fact}[thm]{Fact}

\theoremstyle{remark}
\newtheorem{rem}[thm]{Remark}
\newtheorem{examp}[thm]{Example}

\newcommand{\rr}{\mathbb{R}}
\newcommand{\nn}{\mathbb{N}}

\newcommand{\meg}{\geqslant}
\newcommand{\mik}{\leqslant}
\newcommand{\ave}{\mathbb{E}}

\begin{document}

\title[Subgaussianity is hereditarily determined]{Subgaussianity is hereditarily determined}

\author{Pandelis Dodos and Konstantinos Tyros}

\address{Department of Mathematics, University of Athens, Panepistimiopolis 157 84, Athens, Greece}
\email{pdodos@math.uoa.gr}

\address{Department of Mathematics, University of Athens, Panepistimiopolis 157 84, Athens, Greece}
\email{ktyros@math.uoa.gr}



\begin{abstract}
Let $n$ be a positive integer, let $\boldsymbol{X}=(X_1,\dots,X_n)$ be a random vector in $\rr^n$ with bounded
entries, and let $(\theta_1,\dots,\theta_n)$ be a vector in $\rr^n$. We show that the subgaussian
behavior of the random variable $\theta_1 X_1+\dots +\theta_n X_n$ is essentially determined by the
subgaussian behavior of the random variables $\sum_{i\in H} \theta_i X_i$ where $H$ is a random subset of $\{1,\dots,n\}$.
\end{abstract}

\maketitle


\section{Introduction} \label{sec1}

\numberwithin{equation}{section}

\subsection{Subgaussianity} \label{subsec1.1}

Recall that a real-valued random variable $X$ is called \textit{subgaussian} if its tails are dominated by
(that is, they decay at least as fast as) the~tails of a gaussian. One of the several equivalent ways to quantify
this property is using the Orlicz norm for the function $\psi_2(x)=e^{x^2}-1$. Specifically, the random
variable $X$ is subgaussian if its Orlicz norm
\begin{equation} \label{e1.1}
\|X\|_{\psi_2} \coloneqq \inf\big\{ s>0: \ave\big[e^{(X/s)^2}\big]\mik 2 \big\}
\end{equation}
is finite.

Next, let $n$ be a positive integer, and let $\boldsymbol{X}=(X_1,\dots,X_n)$ be a \textit{random vector} in $\rr^n$, that is,
$\boldsymbol{X}$ is a finite sequence of real-valued random variables defined on a common probability space.
Also let  $K>0$ and $\boldsymbol{\theta}=(\theta_1,\dots,\theta_n)\in \rr^n$, and recall that the random vector
$\boldsymbol{X}$ is said to be $K$-\textit{subgaussian at the direction}~$\boldsymbol{\theta}$ provided~that
\begin{equation} \label{e1.2}
\| \langle \boldsymbol{\theta},\boldsymbol{X}\rangle \|_{\psi_2} \mik K \|\boldsymbol{\theta}\|_2
\end{equation}
where
\begin{equation} \label{e1.3}
\langle \boldsymbol{\theta},\boldsymbol{X}\rangle= \sum_{i=1}^n \theta_i X_i
\end{equation}
is the inner product of $\boldsymbol{\theta}$ and $\boldsymbol{X}$, and $\|\boldsymbol{\theta}\|_2=(\theta_1^2+\dots+\theta_n^2)^{1/2}$
is the euclidean norm of the vector $\boldsymbol{\theta}$.

\subsection{The problem} \label{subsec1.2}

Let $\boldsymbol{X}=(X_1,\dots,X_n)$ be a random vector with $[-1,1]\text{-valued}$ entries,
and fix $\boldsymbol{\theta}=(\theta_1,\dots,\theta_n)\in \rr^n$. For every subset $H$ of $[n]\coloneqq \{1,\dots,n\}$ let
$\boldsymbol{\theta}_H\in \rr^n$ denote the vector defined by
\begin{equation} \label{e1.4}
\boldsymbol{\theta}_H=(\theta'_1,\dots,\theta'_n)\coloneqq
\begin{cases}
\theta'_i=\theta_i & \text{if } i\in H, \\
\theta'_i=0 & \text{otherwise}.
\end{cases}
\end{equation}
In this paper we address the question whether the subgaussian behavior of the random vector $\boldsymbol{X}$ at the
direction $\boldsymbol{\theta}$ is reflected to (and, conversely, whether it~is characterized~by) the typical subgaussian
behavior of $\boldsymbol{X}$ at the direction $\boldsymbol{\theta}_H$ where $H$ is a random subset of $[n]$ distributed
according to the uniform probability measure on $\{0,1\}^n$ or, more generally, according to the $p$-biased
measure\footnote{The $p$-biased measure $\mu_p$ is defined by $\mu_p(\{H\})=p^{|H|}(1-p)^{n-|H|}$ for every
$H\subseteq [n]$. (Here, and in the rest of this paper, we identify every $H\subseteq [n]$ with its indicator
function $\mathbf{1}_H\in \{0,1\}^n$.)} $\mu_p$ ($0<p<1$).

This question was motivated by a problem in density Ramsey theory; see Subsection \ref{subsec5.2} for more details.
Related questions---though of a somewhat different nature---have been studied in high-dimensional probability
and asymptotic convex geometry (see, \emph{e.g.}, \cite{BN}), as well as in the study of thin sets in harmonic analysis (see \cite{Pi}).
It is important to note that the main point in our approach lies in the fact that, apart from the boundedness condition
on $\boldsymbol{X}$, we make no further assumptions on the distributions of the random variables $X_1,\dots,X_n$
and on their correlation. (This level of generality is actually necessary for certain applications in combinatorics.)

\subsection{Examples} \label{subsec1.3}

At this point it is useful to give examples of bounded random vectors which are
subgaussian at a given direction. For concreteness we will restrict our discussion to the direction
$\boldsymbol{\sigma}=(1,\dots,1)\in\rr^n$, but corresponding examples can be given for any other direction.

Undoubtedly, the most important examples are random vectors with independent entries and, more generally, random
vectors which are bounded martingale difference sequences. Another interesting class of examples consists of Sidon
sets of characters in a compact abelian group $G$. (Here, we view $G$ as a probability space equipped with the
Haar probability measure, and we view every character as a complex-valued random variable on $G$;
see \cite{Pi} for details). Note, however, that all these examples are subgaussian at every direction.

A different---but quite relevant---example is a random vector whose entries exhibit high cancellation.
More precisely, fix a $[-1,1]$-valued random variable $Z$. Assume for simplicity that $n$ is even, say $n=2k$,
and fix a subset $T$~of~$[n]$ with $|T|=k$. We define $\boldsymbol{X}=(X_1,\dots,X_n)$ by setting $X_i=Z$ if $i\in T$,
and $X_i=-Z$ if $i\notin T$. Notice that $\langle \boldsymbol{\sigma},\boldsymbol{X}\rangle=0$,
and so $\boldsymbol{X}$ is $K$-subgaussian at the direction~$\boldsymbol{\sigma}$ for any $K>0$.
On the other hand, observe that $\langle \boldsymbol{\sigma}_T,\boldsymbol{X}\rangle=(n/2) Z$; consequently,
if $\boldsymbol{X}$ is $K$-subgaussian at the direction~$\boldsymbol{\sigma}_T$, then
$K\meg (\|Z\|_{\psi_2}/\sqrt{2})\, n^{1/2}$. Nevertheless, it is easy to see that we may select, with positive probability,
a subset $H$ of $[n]$ such that $\boldsymbol{X}$ is $O(1)$-subgaussian at the direction~$\boldsymbol{\sigma}_{H}$.

All the above examples can be combined together by taking convex combinations. Precisely, let $J$
be a nonempty finite set, and for every $j\in J$ let $\boldsymbol{X}_j$ be a random vector in $\rr^n$ whose entries
are either independent, or exhibit high cancellation in the sense we described above. If $\boldsymbol{X}$
is any convex combination of $(\boldsymbol{X}_j: j\in J)$, then clearly $\boldsymbol{X}$
is $O(1)\text{-subgaussian}$ at the direction $\boldsymbol{\sigma}$, but it is already not quite straightforward
to find a subset $H$ of $[n]$ with $|H|=n/2+O(\sqrt{n})$ such that $\boldsymbol{X}$
is $O(1)$-subgaussian at the direction~$\boldsymbol{\sigma}_{H}$.

\subsection{The main result} \label{subsec1.4}

Our main result shows that such a selection is possible in full generality.
Specifically, we have the following theorem; more precise quantitative versions are given in Proposition \ref{p3.1}
and Theorem \ref{t4.1} in the main text. (For~our conventions for asymptotic notation see Subsection \ref{subsec2.2};
recall that by $\mu_p$ we denote the $p$-biased measure on $\{0,1\}^n$.)
\begin{thm} \label{t1.1}
The following hold.
\begin{enumerate}
\item[(1)] Let $K>0$, and let\, $0<p<1$. Also let $n$ be a positive integer, let $\boldsymbol{X}$ be a random vector in $\rr^n$
with\, $[-1,1]\text{-valued}$ entries, and let $\boldsymbol{\theta}\in\rr^n$. If\, $\boldsymbol{X}$ is $K\text{-subgaussian}$
at the direction $\boldsymbol{\theta}$, then for every $C>0$
\begin{equation} \label{e1.5}
\mu_p\big( \{ H: \boldsymbol{X} \text{ is $C$-subgaussian at the direction } \boldsymbol{\theta}_H\}\big)\meg p-o_{C\to\infty;K,p}(1).
\end{equation}
$($Thus, the error term in \eqref{e1.5} does not dependent on the dimension~$n$,
the random vector $\boldsymbol{X}$, and the direction $\boldsymbol{\theta}$.$)$
\item[(2)] Conversely, let $K>0$, let $0<p<1$, and let $0<\gamma\mik 1$. Also let $n$ be a positive integer, let $\boldsymbol{X}$
be a random vector in\, $\rr^n$ with\, $[-1,1]\text{-valued}$ entries, and let\, $\boldsymbol{\theta}\in \rr^n$. If
\[ \mu_p\big( \{ H: \boldsymbol{X} \text{ is $K$-subgaussian at the direction } \boldsymbol{\theta}_H\}\big)\meg \gamma, \]
then $\boldsymbol{X}$ is $O_{K,p,\gamma}(1)$-subgaussian at the direction $\boldsymbol{\theta}$.
\end{enumerate}
\end{thm}

\subsection{Sharpness of the probability} \label{subsec1.5}

Although the lower bound in \eqref{e1.5} is independent of the
direction $\boldsymbol{\theta}$, we note that the probability appearing on the left-hand side of \eqref{e1.5}
does depend upon the choice of $\boldsymbol{\theta}$. Indeed, if $\boldsymbol{\theta}=(1,\dots,1)\in \rr^n$,
then this probability is $1-o_{C\to\infty;K,p}(1)-o_{n\to\infty;p}(1)$. (See Corollary \ref{c4.11} in the main text.)
At~the other extreme, there exist random vectors and directions in $\rr^n$ for which the corresponding~probability
is at most $p+o_{n\to\infty;p,C}(1)$ for any fixed~$C>0$. (See Example~\ref{ex4.2}.)
In particular, the lower bound in \eqref{e1.5} is optimal.

\subsection{Related results/Outline of the argument} \label{subsec1.6}

Beyond its probabilistic content, Theorem \ref{t1.1} can also be placed in the general context of property testing
(see, \emph{e.g.}, \cite{G}). Indeed, Theorem \ref{t1.1} essentially asserts that subgaussianity, at any given direction, is testable.

Theorem \ref{t1.1} can also be viewed as a partial unconditionality result, in the spirit of the work of
Elton \cite{E1,E2} and Pajor \cite{Pa}. In fact, this is more than an analogy since part (1) of Theorem \ref{t1.1}
for $p=1/2$ and the direction $(1,\dots,1)\in\rr^n$ can be proved using the Sauer--Shelah lemma
which is a main tool in the proof of the Elton--Pajor theorem.

That said, the proof of the general case of Theorem \ref{t1.1} is quite intrinsic and, apart from a couple of basic
tools, it relies exclusively on properties of subgaussian random variables.

The first part is based on a large deviation inequality for the $\psi_2\text{-norm}$ of the random
variables $\langle \boldsymbol{\theta}_H,\boldsymbol{X}\rangle$ which can be seen as a reverse triangle inequality;
this is the content of Proposition \ref{p4.3} in the main text. With this inequality at our disposal,
we detect the behavior of the probability in \eqref{e1.5} using the $\ell_\infty$-norm $\|\boldsymbol{\theta}\|_{\infty}$
of the direction~$\boldsymbol{\theta}$. Specifically, if $\|\boldsymbol{\theta}\|_2=1$ and $\|\boldsymbol{\theta}\|_{\infty}$
is sufficiently small, say $\|\boldsymbol{\theta}\|_{\infty}\mik 1/L$, then we may select $C=O_{K,p,L}(1)$ such that
the corresponding probability is $1-o_{L\to\infty;K}(1)-o_{n\to\infty,p}(1)$. On the other hand, if
$\|\boldsymbol{\theta}\|_{\infty}\meg 1/L$, then we fix a coordinate $i_0\in [n]$ such that $|\theta_{i_0}|\meg 1/L$
and we proceed by conditioning on the set of all $H\subseteq [n]$ such that $i_0\in H$.
\begin{rem} \label{r1.2}
The argument is roughly analogous to the proof of Roth's theorem~\cite{Ro}.
Indeed, the case where the $\ell_\infty$-norm is small corresponds to case of small Fourier bias and it implies
pseudorandomness. On the other hand, the case where the $\ell_\infty$-norm is non-negligible corresponds to the case
of correlation with a character, and the proof takes advantage of this structural information.
\end{rem}
The proof of the second part of Theorem \ref{t1.1} is quite simple, and it follows from a standard application
of the bounded differences inequality.

\subsection{Structure of the paper} \label{subsec1.7}

We close this introduction by briefly discussing the contents of this paper. In Section \ref{sec2},
we fix our notation (which is mostly standard), and we recall some basic material which is needed for the proof of our main result.
In Section \ref{sec3} we give the proof of part (2) of Theorem \ref{t1.1}, and in Section \ref{sec4} we give the proof of part (1).
Finally, in Section \ref{sec5} we present and we comment on various extensions of Theorem~\ref{t1.1}.

\subsubsection*{Acknowledgments} We would like to thank the anonymous referee for carefully reading the paper
and for several helpful suggestions.


\section{Background material} \label{sec2}

\numberwithin{equation}{section}

\subsection{ \ } \label{subsec2.1}

By $\nn=\{0,1,\dots\}$ we denote the set of all natural numbers. Recall that for every positive integer $n$
we set $[n]\coloneqq \{1,\dots,n\}$. Moreover, for every finite set~$H$ by $|H|$ we denote its cardinality.

\subsection{ \ } \label{subsec2.2}

We use the following $o(\cdot)$ and $O(\cdot)$ notation. If $a_1,\dots,a_k$ are parameters
and $C$ is a positive real/integer, then we write $o_{C\to\infty;a_1,\dots,a_k}(X)$ to denote a quantity bounded in magnitude
by $X F_{a_1,\dots,a_k}(C)$ where $F_{a_1,\dots,a_k}$ is a function which depends on $a_1,\dots,a_k$ and goes to zero as $C\to\infty$.
Similarly, by $O_{a_1,\dots,a_k}(X)$ we denote a quantity bounded in magnitude by $X C_{a_1,\dots,a_k}$ where $C_{a_1,\dots,a_k}$
is a positive constant depending on the parameters $a_1,\dots,a_k$.

\subsection{ \ } \label{subsec2.3}

As we have mentioned, for every positive integer $n$ and every $0<p<1$ by $\mu_p$ we denote
the $p$-biased measure on $\{0,1\}^n$, that is, the probability measure on $\{0,1\}^n$ which is defined by setting
\begin{equation} \label{e2.1}
\mu_p(\{H\})= p^{|H|}(1-p)^{n-|H|}
\end{equation}
for every $H\subseteq [n]$. In particular, $\mu_{1/2}$ is the uniform probability measure on $\{0,1\}^n$.

\subsection{ \ } \label{subsec2.4}

For every vector $\boldsymbol{c}=(c_1,\dots,c_n)$ in $\rr^n$ and every $1\mik p\mik \infty$
by $\|\boldsymbol{c}\|_p$ we shall denote the $\ell_p$-norm of $\boldsymbol{c}$, that is,
$\|\boldsymbol{c}\|_p=(|c_1|^p+\dots+|c_n|^p)^{1/p}$ if $1\mik p<\infty$, and
$\|\boldsymbol{c}\|_\infty=\max\{|c_1|,\dots,|c_n|\}$.

\subsection{Properties of subgaussian random variables} \label{subsec2.5}

We will need the following properties of subgaussian random variables.
For a proof, as well as for a detailed discussion of related material, see \cite[Chapter 2]{V}.
\begin{prop} \label{p2.1}
Let $X$ be a real-valued random variable.
\begin{enumerate}
\item[(a)] If\, $X$ is subgaussian, then we have $\mathbb{P}(\{|X|\meg t\})\mik 2\exp(-t^2/\|X\|_{\psi_2}^2)$
for every $t>0$.
\item[(b)] Conversely, let $K>0$ and assume that\, $\mathbb{P}(\{|X|\meg t\})\mik 2\exp(-t^2/K^2)$
for every $t>0$. Then, $X$ is subgaussian and, moreover, $\|X\|_{\psi_2}\mik \sqrt{3}\, K$.
\end{enumerate}
\end{prop}

\subsection{Hoeffding's inequality and the bounded differences inequality} \label{subsec2.6}

In various places in the paper, we will apply Hoeffding's inequality and the bounded differences inequality.
We will use these basic inequalities in a form which, although less general, is better suited to our needs.
(The standard forms of these inequalities and their proofs can be found, \emph{e.g.}, in \cite[Theorem 2.8]{BLM}
and \cite[Theorem~6.2]{BLM} respectively.)

Precisely, we will need the following consequence of Hoeffding's inequality.
\begin{prop} \label{p2.2}
Let $n$ be a positive integer, and let $\boldsymbol{c}=(c_1,\dots,c_n)\in \rr^n\setminus\{0\}$. Also let $0<p<1$.
Then for any $t>0$ we have
\begin{equation} \label{e2.2}
\mu_p\Big(\Big\{ H : \Big| \sum_{i\in H}c_i -p \sum_{i=1}^n c_i \Big| \meg t \Big\}\Big) \mik
2\exp\Big(-\frac{2t^2}{\|\boldsymbol{c}\|_2^2}\Big).
\end{equation}
\end{prop}
We will also need the following special case of the bounded differences inequality.
\begin{prop} \label{p2.3}
Let $n$ be a positive integer, let $f\colon \{0,1\}^n\to\rr$ be a function, and let $\boldsymbol{c}=(c_1,\dots,c_n)\in \rr^n\setminus\{0\}$
such that for every $i\in [n]$ and every $H\subseteq [n]\setminus \{i\}$
\begin{equation} \label{e2.3}
|f(H\cup\{i\})-f(H)|\mik c_i.
\end{equation}
Also let $0<p<1$. Then, setting $M\coloneqq \underset{H\sim \mu_p}{\ave} f(H)$, for any $t>0$ we have
\begin{equation} \label{e2.4}
\mu_p\big( \big\{H: | f(H)-M|\meg t\big\}\big)\mik 2\exp\Big(-\frac{2t^2}{\|\boldsymbol{c}\|^2_2}\Big).
\end{equation}
\end{prop}


\section{Proof of Theorem \ref{t1.1}: part (2)} \label{sec3}

\numberwithin{equation}{section}

We have the following, more informative, version of part (2) of Theorem \ref{t1.1}.
\begin{prop} \label{p3.1}
Let $K>0$, let $0<p<1$, let $0<\gamma\mik 1$, and set
\begin{equation} \label{e3.1}
C=C(K,p,\gamma)\coloneqq p^{-1}\, \big(K+\sqrt{\ln(2/\gamma)}\big).
\end{equation}
Also let $n$ be a positive integer, let $\boldsymbol{X}=(X_1,\dots,X_n)$ be a random vector in $\rr^n$ with
$\|X_i\|_{\psi_2}\mik 1$ for every $i\in [n]$, and let\, $\boldsymbol{\theta}\in \rr^n\setminus \{0\}$. If
\begin{equation} \label{e3.2}
\mu_p\big( \{ H: \boldsymbol{X} \text{ is $K$-subgaussian at the direction } \boldsymbol{\theta}_H\}\big)\meg \gamma,
\end{equation}
then $\boldsymbol{X}$ is $C$-subgaussian at the direction $\boldsymbol{\theta}$.
\end{prop}
\begin{rem} \label{r3.2}
We do not know which is the optimal dependence of the constant $C(K,p,\gamma)$ with respect to the parameters $K,p$ and $\gamma$.
The referee noted that the dependence on $p$ could be improved; observe that the parameter $p$ is important in the sparse regime,
that is, when $p=o_{n\to\infty}(1)$.
\end{rem}
Proposition \ref{p3.1} is based on two auxiliary results. The first one is an elementary identity
which expresses the random variable $\langle \boldsymbol{\theta},\boldsymbol{X}\rangle$ as a linear combination of the
random variables $\langle \boldsymbol{\theta}_H,\boldsymbol{X}\rangle$.
\begin{fact} \label{f3.3}
Let $p,n,\boldsymbol{X},\boldsymbol{\theta}$ be as in Proposition \emph{\ref{p3.1}}. Then we have
\begin{equation} \label{e3.3}
\langle \boldsymbol{\theta},\boldsymbol{X}\rangle = p^{-1} \sum_{H\subseteq [n]} \mu_p(\{H\})\,
\langle \boldsymbol{\theta}_H,\boldsymbol{X}\rangle.
\end{equation}
In particular,
\begin{equation} \label{e3.4}
\|\langle \boldsymbol{\theta},\boldsymbol{X}\rangle \|_{\psi_2} \mik p^{-1}\,\,
\underset{H\sim \mu_p}{\ave} \|\langle \boldsymbol{\theta}_H,\boldsymbol{X}\rangle \|_{\psi_2}.
\end{equation}
\end{fact}
\begin{proof}
Observe that
\[ \sum_{H\subseteq [n]} \mu_p(\{H\})\, \langle \boldsymbol{\theta}_H,\boldsymbol{X}\rangle =
\sum_{i=1}^n \theta_i X_i \Big( \sum_{i\in H\subseteq [n]} \mu_p(\{H\})\Big) = p\, \langle \boldsymbol{\theta},\boldsymbol{X}\rangle. \]
The estimate in \eqref{e3.4} follows from this identity and the triangle inequality.
\end{proof}
The second auxiliary result is the following, fairly straightforward, consequence of the bounded differences inequality;
we isolate this consequence for future use.
\begin{lem} \label{l3.4}
Let $p,n,\boldsymbol{X},\boldsymbol{\theta}$ be as in Proposition \emph{\ref{p3.1}}. Then, setting
\begin{equation} \label{e3.5}
M\coloneqq \underset{H\sim \mu_p}{\ave} \|\langle \boldsymbol{\theta}_H,\boldsymbol{X}\rangle \|_{\psi_2},
\end{equation}
for any $t>0$ we have
\begin{equation} \label{e3.6}
\mu_p\big( \big\{H: \big| \|\langle \boldsymbol{\theta}_H,\boldsymbol{X}\rangle \|_{\psi_2}-M\big|\meg t\big\}\big)\mik
2\exp\Big(-\frac{2t^2}{\|\boldsymbol{\theta}\|^2_2}\Big).
\end{equation}
\end{lem}
\begin{proof}
By the triangle inequality, for every $i\in [n]$ and every $H\subseteq [n]\setminus \{i\}$ we have
\[ \big| \|\langle \boldsymbol{\theta}_{H\cup\{i\}},\boldsymbol{X}\rangle \|_{\psi_2}-
\|\langle \boldsymbol{\theta}_H,\boldsymbol{X}\rangle \|_{\psi_2}\big| \mik \|\theta_i X_i\|_{\psi_2} \mik \theta_i.\]
Using this observation, the result follows from Proposition \ref{p2.3}.
\end{proof}
We are now ready to proceed to the proof of Proposition \ref{p3.1}.
\begin{proof}[Proof of Proposition \emph{\ref{p3.1}}]
Setting $t_0\coloneqq \sqrt{\ln(2/\gamma)}\,\, \|\boldsymbol{\theta}\|_2>0$, by \eqref{e3.6}, we have
\[ \mu_p\big( \big\{H: \big| \|\langle \boldsymbol{\theta}_H,\boldsymbol{X}\rangle \|_{\psi_2}-M\big|\meg t_0\big\}\big)\mik
\frac{\gamma}{2}. \]
Thus, by \eqref{e3.2}, we may select $H\subseteq [n]$ such that
\begin{enumerate}
\item[$\bullet$] $\|\langle \boldsymbol{\theta}_H,\boldsymbol{X}\rangle \|_{\psi_2}\mik K \|\boldsymbol{\theta}_H\|_2
\mik K\|\boldsymbol{\theta}\|_2$, and
\item[$\bullet$] $M\mik \|\langle \boldsymbol{\theta}_H,\boldsymbol{X}\rangle \|_{\psi_2} +t_0$.
\end{enumerate}
Therefore, $M\mik K\|\boldsymbol{\theta}\|_2+t_0$. By \eqref{e3.4}, \eqref{e3.5} and the choice of $C$ in \eqref{e3.1}, we conclude that
$\|\langle \boldsymbol{\theta},\boldsymbol{X}\rangle \|_{\psi_2}\mik C\|\boldsymbol{\theta}\|_2$, as desired.
\end{proof}


\section{Proof of Theorem \ref{t1.1}: part (1)} \label{sec4}

\numberwithin{equation}{section}

\subsection{ \ } \label{subsec4.1}

This section is devoted to the proof of the following theorem.
\begin{thm} \label{t4.1}
Let $K>0$, let $0<p<1$, let $0< \eta <p$, and set
\begin{equation} \label{e4.1}
C=C(K,p,\eta)\coloneqq 18\, \frac{(K+1)}{p} \, \log_2\Big(\frac{4}{\eta}\Big).
\end{equation}
Also let $n$ be a positive integer, let $\boldsymbol{X}$ be a random vector in $\rr^n$ with $[-1,1]\text{-valued}$ entries,
and let $\boldsymbol{\theta}\in\rr^n\setminus \{0\}$. If\, $\boldsymbol{X}$ is $K\text{-subgaussian}$
at the direction~$\boldsymbol{\theta}$, then
\begin{equation} \label{e4.2}
\mu_p\big( \{ H: \boldsymbol{X} \text{ is $C$-subgaussian at the direction } \boldsymbol{\theta}_H\}\big)\meg p-\eta.
\end{equation}
\end{thm}
It is clear that Theorem \ref{t4.1} yields part (1) of Theorem \ref{t1.1}. As we have already pointed out in the introduction,
the lower bound in \eqref{e4.2} is optimal.
\begin{examp} \label{ex4.2}
Let $n$ be an arbitrary positive integer, and set
\begin{equation} \label{e4.3}
\boldsymbol{\theta}\coloneqq (n, \underbrace{1,\dots,1}_{n\mathrm{-times}})\in \rr^{n+1}.
\end{equation}
We fix a $[-1,1]$-valued random variable $Z$ and, as in Subsection \ref{subsec1.3}, we define the (high cancellation)
random vector $\boldsymbol{X}=(X_1,\dots,X_{n+1})$ in $\rr^{n+1}$ by setting $X_1=-Z$, and $X_i=Z$
if $i\in \{2,\dots,n+1\}$. Since $\langle\boldsymbol{\theta},\boldsymbol{X}\rangle=0$, the random vector
$\boldsymbol{X}$ is $K\text{-subgaussian}$ at the direction $\boldsymbol{\theta}$ for any $K>0$. Next,
let $0<p<1$ be arbitrary, and set
\[ \mathcal{H}\coloneqq \{H\subseteq [n+1]:1\notin H \text{ and } |H\cap\{2,\dots,n+1\}|\meg pn/2\}. \]
By~Proposition \ref{p2.2}, we see that $\mu_p(\mathcal{H})= 1-p-o_{n\to\infty;p}(1)$.
Moreover, if $H\in\mathcal{H}$, then $\langle\boldsymbol{\theta}_H,\boldsymbol{X}\rangle=|H|\, Z$ and, therefore,
if $K$ is any positive real such that $\boldsymbol{X}$ is $K\text{-subgaussian}$ at the direction $\boldsymbol{\theta}_H$,
then $K\meg (\sqrt{p/2} \,\, \|Z\|_{\psi_2})\, n^{1/2}$. Thus, we conclude that for any $C>0$,
\begin{equation} \label{e4.4}
\mu_p\big( \{ H: \boldsymbol{X} \text{ is $C$-subgaussian at the direction } \boldsymbol{\theta}_H\}\big)\mik p+o_{n\to\infty;p,C}(1).
\end{equation}
\end{examp}

\subsection{A large deviation inequality for the $\psi_2$-norm} \label{subsec4.2}

The first step of the proof of Theorem \ref{t4.1} is the following large deviation inequality.
\begin{prop} \label{p4.3}
Let $K\meg 1/\sqrt{2}$, and let $0<p<1$. Also let $n, \boldsymbol{X},\boldsymbol{\theta}$ be as in Theorem \emph{\ref{t4.1}}.
If $\boldsymbol{X}$ is $K\text{-subgaussian}$ at the direction $\boldsymbol{\theta}$, then for any $\lambda\meg 8\sqrt{2}$ we~have
\begin{equation} \label{e4.5}
\mu_p\big( \big\{ H: \|\langle\boldsymbol{\theta}_H,\boldsymbol{X}\rangle\|_{\psi_2} \meg
\lambda K\|\boldsymbol{\theta}\|_2 \big\}\big) \mik 3\exp\Big(-\frac{\ln 2}{32}\lambda^2\Big).
\end{equation}
\end{prop}
In order to put Proposition \ref{p4.3} in a proper context recall that, by \eqref{e3.3} and the triangle
inequality, we have $p\, \|\langle \boldsymbol{\theta},\boldsymbol{X}\rangle \|_{\psi_2} \mik
\underset{H\sim \mu_p}{\ave} \|\langle \boldsymbol{\theta}_H,\boldsymbol{X}\rangle \|_{\psi_2}$.
The next corollary shows that this estimate can actually be \textit{reversed}. Thus, we may
view Proposition~\ref{p4.3} as a reverse triangle inequality.
\begin{cor} \label{c4.4}
Let $K\meg 1/\sqrt{2}$, and let $0<p<1$. Also let $n, \boldsymbol{X},\boldsymbol{\theta}$ be as in Theorem \emph{\ref{t4.1}}.
If $\boldsymbol{X}$ is $K\text{-subgaussian}$ at the direction $\boldsymbol{\theta}$, then
\begin{equation} \label{e4.6}
\underset{H\sim \mu_p}{\ave} \|\langle \boldsymbol{\theta}_H,\boldsymbol{X}\rangle \|_{\psi_2}\mik
12 K\,\|\boldsymbol{\theta}\|_2.
\end{equation}
In particular, if\, $\|\langle \boldsymbol{\theta},\boldsymbol{X}\rangle \|_{\psi_2}\meg \|\boldsymbol{\theta}\|_2/\sqrt{2}$, then
\begin{equation} \label{e4.7}
\underset{H\sim \mu_p}{\ave} \|\langle \boldsymbol{\theta}_H,\boldsymbol{X}\rangle \|_{\psi_2}\mik
12\, \|\langle \boldsymbol{\theta},\boldsymbol{X}\rangle \|_{\psi_2}.
\end{equation}
\end{cor}
\begin{proof}
It is a straightforward consequence of Proposition \ref{p4.3}. Indeed,
\begin{eqnarray*}
\underset{H\sim \mu_p}{\ave} \|\langle \boldsymbol{\theta}_H,\boldsymbol{X}\rangle \|_{\psi_2}
& = & \int_{0}^{\infty} \mu_p\big(\big\{H: \|\langle \boldsymbol{\theta}_H,\boldsymbol{X}\rangle \|_{\psi_2}\meg t\big\}\big)\, dt \\
& = &  K\|\boldsymbol{\theta}\|_2 \int_{0}^{\infty} \mu_p\big(\big\{ H:
\|\langle \boldsymbol{\theta}_H,\boldsymbol{X}\rangle \|_{\psi_2}\meg \lambda K\|\boldsymbol{\theta}\|_2\big\}\big)\, d\lambda \\
& \mik & K\|\boldsymbol{\theta}\|_2 \Big(8\sqrt{2}+ \int_{8\sqrt{2}}^{\infty} 3 \exp\Big(-\frac{\ln 2}{32}\lambda^2\Big)\, d\lambda\Big)
\mik 12K\|\boldsymbol{\theta}\|_2
\end{eqnarray*}
as desired.
\end{proof}
Corollary \ref{c4.4} can be used, in turn, to upgrade Proposition \ref{p4.3} and provide finer information for the distribution of the
$\psi_2$-norm of the random variables  $\langle \boldsymbol{\theta}_H,\boldsymbol{X}\rangle$. Specifically, we have the following corollary;
it follows immediately by Lemma~\ref{l3.4}, Corollary  \ref{c4.4}, and taking into account the fact that $\|X\|_{\psi_2}\mik 1/\sqrt{\ln 2}$
for every $[-1,1]$-valued random variable $X$.
\begin{cor} \label{c4.5}
Let $K\meg 1/\sqrt{2}$, and let $0<p<1$. Also let $n, \boldsymbol{X},\boldsymbol{\theta}$ be as in Theorem \emph{\ref{t4.1}}.
If $\boldsymbol{X}$ is $K\text{-subgaussian}$ at the direction $\boldsymbol{\theta}$, then for any $\lambda> 0$
\begin{equation} \label{e4.8}
\mu_p\big(\big\{H: \|\langle \boldsymbol{\theta}_H,\boldsymbol{X}\rangle \|_{\psi_2} \meg (12+\lambda)K\|\boldsymbol{\theta}\|_2 \big\}\big)
\mik 2\exp(-2\ln 2\,\lambda^2K^2).
\end{equation}
\end{cor}

\subsection{Proof of Proposition \ref{p4.3}} \label{subsec4.3}

It is based on the following lemma.
\begin{lem} \label{l4.6}
Let $K>0$, and let $0<p<1$. Also let $n,\boldsymbol{X}, \boldsymbol{\theta}$ be as in Theorem~\emph{\ref{t4.1}}.
If $\boldsymbol{X}$ is $K\text{-subgaussian}$ at the direction $\boldsymbol{\theta}$, then, setting
$Q\coloneqq \max\{2pK,\sqrt{2}\}$, for every $M\meg \max\{4\sqrt{2\ln 2}\, pK , 4\sqrt{\ln 2}\}$ we have
\begin{equation} \label{e4.9}
\mu_p\Big( \Big\{ H: \| \langle \boldsymbol{\theta}_H, \boldsymbol{X} \rangle \|_{\psi_2} \mik
\sqrt{\frac{3}{\ln 2}}\,M\|\boldsymbol{\theta}\|_2 \Big\}\Big) \meg 1- 3 \exp\Big(-\frac{M^2}{2Q^2}\Big).
\end{equation}
\end{lem}
It is easy to see that Proposition \ref{p4.3} follows from Lemma \ref{l4.6}. Indeed, let $\lambda\meg 8\sqrt{2}$
be arbitrary, and set $M\coloneqq (\sqrt{\ln 2}/2)\,\lambda K$. It is easy to see that with this choice we have that
$M\meg \max\{4\sqrt{2\ln 2}\,pK , 4\sqrt{\ln 2}\}$. Noticing that $2K\meg\max\{2pK,\sqrt{2}\}$,
by Lemma \ref{l4.6}, we conclude that \eqref{e4.5} is satisfied.

Thus, it is enough to prove Lemma \ref{l4.6}. To this end, we need the following sublemma.
\begin{sbl} \label{s4.7}
Let $X$ be a real-valued random variable, let $R, C>0$, and assume that $\mathbb{P}(\{|X|\meg 2^j R\})\mik 2 \exp(-(2^j R)^2 /C^2)$
for every $j\in\nn$. Then we have
\[ \|X\|_{\psi_2}\mik\sqrt{3}\max\{2C,R/\sqrt{\ln 2}\}. \]
\end{sbl}
\begin{proof}
Set $N\coloneqq\max\{2C,R/\sqrt{\ln 2}\}$. By Proposition \ref{p2.1}, it suffices to show that for every $t>0$ we have
$\mathbb{P}(\{|X|\meg t\})\mik 2\exp(-t^2/N^2)$.

Indeed, notice first that, since $N\meg R/\sqrt{\ln 2}$, we have $2\exp(-R^2/N^2)\meg1$. This, in turn,
implies that $\mathbb{P}(\{|X|\meg t\})\mik 2\exp(-t^2/N^2)$ if $0<t\mik R$.

The remaining cases (that is, when $t\meg R$) follow from our hypothesis and a standard dyadic pigeonholing.
Specifically, for every $j\in\nn$ set $t_j\coloneqq 2^j R$ and observe that
\begin{equation} \label{e4.10}
\mathbb{P}(\{|X| \meg t_j\})\mik 2 \exp(-t_j^2/C^2).
\end{equation}
Let $t\meg R$ be arbitrary and let $j_0\in\nn$ be such that $t_{j_0} \mik t <t_{j_0+1}=2t_{j_0}$. Then we have
\[ \mathbb{P}(\{|X|\meg t\}) \mik \mathbb{P}(\{|X|\meg t_{j_0}\}) \mik 2 \exp(-t_{j_0}^2/C^2)
\mik  2 \exp(-t^2/(2C)^2) \mik 2 \exp(-t^2/N^2) \]
and the proof is completed.
\end{proof}
We are ready to proceed to the proof of Lemma \ref{l4.6}.
\begin{proof}[Proof of Lemma \emph{\ref{l4.6}}]
The left-hand side of \eqref{e4.9} is scale-invariant; thus we may assume that $\|\boldsymbol{\theta}\|_2=1$, and it is enough to prove that
\begin{equation} \label{e4.11}
\mu_p\Big( \Big\{ H: \| \langle \boldsymbol{\theta}_H, \boldsymbol{X} \rangle \|_{\psi_2} \mik
\sqrt{\frac{3}{\ln 2}}\,M \Big\}\Big) \meg 1- 3 \exp\Big(-\frac{M^2}{2Q^2}\Big)
\end{equation}
for every $M\meg \max\{4\sqrt{2\ln 2}\, pK , 4\sqrt{\ln 2}\}$.
\medskip

\noindent \textit{Step 1.} We will show that for every $t>0$ we have
\begin{equation} \label{e4.12}
\mu_p\Big(\Big\{ H: \mathbb{P}\big(\big\{|\langle \boldsymbol{\theta}_H,\boldsymbol{X} \rangle|\meg t\big\}\big) \mik
2 \exp\Big(-\frac{t^2}{2Q^2}\Big)\Big\}\Big) \meg 1- 2 \exp\Big(-\frac{t^2}{2Q^2}\Big).
\end{equation}
Fix $t>0$. Let $(\Omega,\mathcal{F},\mathbb{P})$ denote the underlying probability space. Let $\omega\in\Omega$ be arbitrary;
since $\boldsymbol{X}(\omega)\in [-1,1]^n$ and $\|\boldsymbol{\theta}\|_2=1$, by Proposition \ref{p2.2}, we have
\begin{equation} \label{e4.13}
\mu_p \Big(\Big\{ H : \big| \langle \boldsymbol{\theta}_H,\boldsymbol{X}(\omega)\rangle -
p\, \langle \boldsymbol{\theta}, \boldsymbol{X}(\omega)\rangle \big|< \frac{t}{2} \Big\}\Big)
\meg 1 - 2 \exp\Big(-\frac{t^2}{2}\Big).
\end{equation}
(We note that here is the only place in the argument where the boundedness of the random vector $\boldsymbol{X}$ is used.)
Next, observe that the event
\begin{equation} \label{e4.14}
\big\{ (H,\omega): |\langle \boldsymbol{\theta}_H,\boldsymbol{X}(\omega)\rangle| <t \big\}
\end{equation}
contains the event
\begin{equation} \label{e4.15}
\Big\{ (H,\omega): |\langle \boldsymbol{\theta},\boldsymbol{X}(\omega)\rangle|<\frac{t}{2p} \text{ and }
\big|\langle \boldsymbol{\theta}_H, \boldsymbol{X}(\omega)\rangle -
p\,\langle \boldsymbol{\theta}, \boldsymbol{X}(\omega)\rangle\big|< \frac{t}{2}\Big\}.
\end{equation}
Finally, notice that $\|\langle \boldsymbol{\theta},\boldsymbol{X} \rangle\|_{\psi_2}\mik K$ since $\|\boldsymbol{\theta}\|_2=1$
and $\boldsymbol{X}$ is $K$-subgaussian at the direction $\boldsymbol{\theta}$. Thus, by Proposition \ref{p2.1} applied to the fixed $t$, we have
\begin{equation} \label{e4.16}
\mathbb{P}\Big( \Big\{ |\langle \boldsymbol{\theta},\boldsymbol{X}\rangle|<\frac{t}{2p} \Big\}\Big)\meg
1-2\exp\Big(-\frac{t^2}{(2pK)^2}\Big).
\end{equation}
Let $\mu_p\times\mathbb{P}$ denote the product probability measure of $\mu_p$ and $\mathbb{P}$. Then using: (i) the estimates
in \eqref{e4.13} and \eqref{e4.16}, (ii) the inclusion of the events in \eqref{e4.14} and \eqref{e4.15}, (iii) the choice of the
constant $Q$, and (iv) Fubini's theorem, we obtain that
\begin{equation} \label{e4.17}
\mu_p\times\mathbb{P} \big(\big\{ (H,\omega) : |\langle \boldsymbol{\theta}_H,\boldsymbol{X}(\omega)\rangle|<t\big\}\big)
\meg 1 - 4\exp\Big(-\frac{t^2}{Q^2}\Big)
\end{equation}
or, equivalently,
\begin{equation} \label{e4.18}
\mu_p\times\mathbb{P} \big(\big\{ (H,\omega) : |\langle \boldsymbol{\theta}_H,\boldsymbol{X}(\omega)\rangle| \meg t\big\}\big)
\mik 4\exp\Big(-\frac{t^2}{Q^2}\Big).
\end{equation}
By \eqref{e4.18} and Markov's inequality, we conclude that
\begin{equation} \label{e4.19}
\mu_p\Big(\Big\{ H: \mathbb{P}\big(\big\{|\langle \boldsymbol{\theta}_H,\boldsymbol{X} \rangle|\meg t\big\}\big) >
2 \exp\Big(-\frac{t^2}{2Q^2}\Big)\Big\}\Big) \mik 2 \exp\Big(-\frac{t^2}{2Q^2}\Big)
\end{equation}
which is clearly equivalent to \eqref{e4.12}.
\medskip

\noindent \textit{Step 2.} We will estimate the probability in \eqref{e4.11} using a discretization argument,  \eqref{e4.12}
and Sublemma \ref{s4.7}. We proceed to the details.

Let $M\meg \max\{4\sqrt{2\ln 2}\,pK , 4\sqrt{\ln 2}\}$ be arbitrary. For every $j\in\nn$ set
\begin{equation} \label{e4.20}
\mathcal{C}_M^j\coloneqq \Big\{H: \mathbb{P}\big(\big\{ |\langle\boldsymbol{\theta}_H,\boldsymbol{X}\rangle|\meg 2^j M\big\}\big)
\mik 2\exp\Big(-\frac{2^{2j}M^2}{2Q^2}\Big)\Big\}
\end{equation}
and observe that, by \eqref{e4.12}, we have  $\mu_p(\mathcal{C}_M^j)\meg 1- 2\exp\Big(-\frac{2^{2j}M^2}{2Q^2}\Big)$.
Therefore, setting
\begin{equation} \label{e4.21}
\mathcal{C}_M \coloneqq \bigcap_{j\in\nn} \mathcal{C}_M^j,
\end{equation}
we have
\begin{eqnarray} \label{e4.22}
\mu_p(\mathcal{C}_M) & \meg & 1-2\sum_{j=0}^{\infty} \exp\Big(-\frac{2^{2j}M^2}{2Q^2}\Big) \\
& \meg & 1- 2\exp\Big(-\frac{M^2}{2Q^2}\Big) - 2\sum_{j=1}^{\infty} \exp\Big(-\frac{2j M^2}{Q^2}\Big) \nonumber \\
& = & 1- 2 \exp\Big(-\frac{M^2}{2Q^2}\Big) - 2\frac{\exp(-2M^2/Q^2)}{1-\exp(-2M^2/Q^2)} \nonumber \\
& \meg & 1- 3 \exp(-M^2/2Q^2) \nonumber
\end{eqnarray}
where the last inequality holds true since $M\meg 2\sqrt{2\ln 2}\, Q \meg \sqrt{2\ln 2}\, Q$. Moreover,
for every $H\in \mathcal{C}_M$, by Sublemma \ref{s4.7} applied for ``$X = \langle\boldsymbol{\theta}_H,\boldsymbol{X}\rangle$"\!,
``$R=M$" and ``$C=\sqrt{2}\,Q$" and using again the fact that $M\meg 2\sqrt{2\ln 2}\, Q$, we see that
$\|\langle\boldsymbol{\theta}_H,\boldsymbol{X}\rangle\|_{\psi_2}\mik \sqrt{3/\ln 2}\,M$. This shows that \eqref{e4.11} is satisfied,
and the proof of Lemma \ref{l4.6} is completed.
\end{proof}

\subsection{The main dichotomy}  \label{subsec4.4}

The next, and last, step of the proof of Theorem \ref{t4.1} is the following proposition which relates
the probability on the left-hand side of~\eqref{e4.2} with the $\ell_\infty$-norm of the direction $\boldsymbol{\theta}$.
In particular, this probability gets bigger as $\|\boldsymbol{\theta}\|_\infty$ gets smaller.
\begin{prop} \label{p4.8}
Let $K\meg 1/\sqrt{2}$, and let\, $0<p<1$. Also let $n, \boldsymbol{X},\boldsymbol{\theta}$ be as in Theorem~\emph{\ref{t4.1}}.
Assume that $\|\boldsymbol{\theta}\|_2=1$ and that $\boldsymbol{X}$ is $K\text{-subgaussian}$ at the direction~$\boldsymbol{\theta}$.
Finally, let $0<\alpha\mik 1$. Then, for every $\lambda>0$, the following hold.
\begin{enumerate}
\item [(i)] If\, $\|\boldsymbol{\theta}\|_\infty\mik \alpha$, then
\begin{equation} \label{e4.23}
\begin{split}
\mu_p \Big(\big\{ H: \boldsymbol{X} & \text{ is $\big(\sqrt{2/p}\, (12+\lambda)K\big)$-subgaussian at the direction } \boldsymbol{\theta}_H\big\}\Big) \\
& \meg 1- 2\exp(-2\ln 2\, \lambda^2K^2) -  2\exp\Big(-\frac{p^2}{2\alpha^2}\Big).
\end{split}
\end{equation}
\item [(ii)] If\, $\|\boldsymbol{\theta}\|_\infty\meg \alpha$, then
\begin{equation} \label{e4.24}
\begin{split}
\mu_p \Big(\big\{ H: \boldsymbol{X} & \text{ is $\big( (12+\lambda)K\alpha^{-1}\big)$-subgaussian at the direction } \boldsymbol{\theta}_H\big\}\Big) \\
& \meg p - 2\exp(-2\ln 2\, \lambda^2K^2).
\end{split}
\end{equation}
\end{enumerate}
\end{prop}
\begin{rem} \label{r4.9}
Note that the lower bound in \eqref{e4.23} depends upon the choice of $\alpha$ (thus, it is not uniform) but this is offset
by making the subgaussianity constant of $\boldsymbol{X}$ at the direction $\boldsymbol{\theta}_H$ independent of $\alpha$.
In \eqref{e4.24}, this phenomenon is reversed.
\end{rem}
\begin{rem} \label{r4.10}
The dependence on $p$\, in \eqref{e4.23} is tight up to a logarithmic factor. This can be seen by considering the diagonal
direction of a random vector $\boldsymbol{X}$ whose entries are truncated independent exponential random variables.
We are grateful to the referee for pointing this out.
\end{rem}
\begin{proof}[Proof of Proposition \emph{\ref{p4.8}}]
Fix $\lambda >0$, and set
\begin{equation} \label{e4.25}
\mathcal{H}_1\coloneqq \big\{ H: \|\langle\boldsymbol{\theta}_H, \boldsymbol{X}\rangle\|_{\psi_2} \mik (12+\lambda)K \big\}.
\end{equation}
Since $\|\boldsymbol{\theta}\|_2=1$ and $\boldsymbol{X}$ is $K\text{-subgaussian}$ at the direction $\boldsymbol{\theta}$,
by Corollary \ref{c4.5}, we~have
\begin{equation} \label{e4.26}
\mu_p(\mathcal{H}_1)\meg 1 - 2\exp(-2\ln 2\, \lambda^2K^2).
\end{equation}
Also write $\boldsymbol{\theta}=(\theta_1,\dots,\theta_n)$.
\medskip

\noindent \textit{Part} (i): Assume that $\|\boldsymbol{\theta}\|_\infty \mik \alpha$, and set
\begin{equation} \label{e4.27}
\mathcal{H}_2\coloneqq \big\{H: \|\boldsymbol{\theta}_H\|_2> \sqrt{p/2}\big\}.
\end{equation}
Notice that for every $H\in\mathcal{H}_1\cap\mathcal{H}_2$ we have
$\|\langle\boldsymbol{\theta}_H,\boldsymbol{X}\rangle\|_{\psi_2} \mik \sqrt{2/p}\, (12+\lambda)K\,\|\boldsymbol{\theta}_H\|_2$, that is,
the random vector $\boldsymbol{X}$ is $\big(\sqrt{2/p}\, (12+\lambda)K\big)$-subgaussian at the direction~$\boldsymbol{\theta}_H$.
Also observe that
\begin{equation} \label{e4.28}
\frac{1}{\|\boldsymbol{\theta}\|_4^4} = \frac{\|\boldsymbol{\theta}\|_2^2}{\|\boldsymbol{\theta}\|_4^4} =
\frac{\|\boldsymbol{\theta}/\|\boldsymbol{\theta}\|_\infty\|_2^2}
 {\|\boldsymbol{\theta}/\|\boldsymbol{\theta}\|_\infty\|_4^4\cdot\|\boldsymbol{\theta}\|_\infty^2}
\meg\frac{1}{\|\boldsymbol{\theta}\|_\infty^2}\meg\frac{1}{\alpha^2}.
 \end{equation}
Thus, by Proposition \ref{p2.2} applied for the vector ``$\boldsymbol{c}=(\theta_1^2,\dots,\theta_n^2)$" and
``$t = p/2$"\!, we obtain that
\begin{equation} \label{e4.29}
\mu_p(\mathcal{H}_2)\meg 1-2\exp\Big(-\frac{p^2}{2\|\boldsymbol{\theta}\|_4^4}\Big)
\stackrel{\eqref{e4.28}}{\meg} 1-2\exp\Big(-\frac{p^2}{2\alpha^2}\Big).
\end{equation}
Combining \eqref{e4.26} and \eqref{e4.29}, we see that \eqref{e4.23} is satisfied.
\medskip

\noindent \textit{Part} (ii): Now assume that $\|\boldsymbol{\theta}\|_\infty\meg \alpha$. Fix $i_0\in [n]$ such that
$|\theta_{i_0}|\meg\alpha$, and set
\begin{equation} \label{e4.30}
\mathcal{H}_3=\{H: i_0\in H\}.
\end{equation}
Observe that for every $H\in\mathcal{H}_3$ we have $\alpha\mik \|\boldsymbol{\theta}_H\|_\infty\mik \|\boldsymbol{\theta}_H\|_2$.
Consequently, for every $H\in\mathcal{H}_1\cap\mathcal{H}_3$ the random vector $\boldsymbol{X}$ is
$\big((12+\lambda)K\alpha^{-1}\big)$-subgaussian at the direction $\boldsymbol{\theta}_H$.
Since $\mu_p(\mathcal{H}_3)=p$, the result follows.
\end{proof}
We close this subsection with the following consequence of Proposition \ref{p4.8} which complements Example \ref{ex4.2}
and concerns the behavior of the probability in \eqref{e4.2} for the ``flat" vector $(1,\dots,1)\in\rr^n$.
\begin{cor} \label{c4.11}
Let $K>0$, and let\, $0< p < 1$. Also let $n, \boldsymbol{X}$ be as in Theorem~\emph{\ref{t4.1}}, and set\,
$\boldsymbol{\sigma}\coloneqq (1,\dots,1)\in\rr^n$. If\, $\boldsymbol{X}$ is $K\text{-subgaussian}$
at the direction $\boldsymbol{\sigma}$, then for every $\lambda>0$ we have
\begin{equation} \label{e4.31}
\begin{split}
\mu_p \Big(\big\{ H: \boldsymbol{X} & \text{ is $\big(\sqrt{2/p}\, (12+\lambda)(K+1)\big)$-subgaussian at the direction }
\boldsymbol{\sigma}_{H}\big\}\Big) \\
& \meg 1- 2\exp(-2\ln 2\, \lambda^2(K+1)^2) -  2\exp(-p^2n/2).
\end{split}
\end{equation}
\end{cor}
\begin{proof}
It follows by part (i) of Proposition \ref{p4.8} applied to the vector ``$\boldsymbol{\theta}=\boldsymbol{\sigma}/\sqrt{n}$"
(notice that $\|\boldsymbol{\theta}\|_2=1$), the constant ``$K=K+1$" and ``$\alpha=1/\sqrt{n}$\,"\!.
\end{proof}

\subsection{Proof of Theorem \ref{t4.1}} \label{subsec4.5}

The result follows by applying Proposition \ref{p4.8} for
\begin{equation} \label{e4.32}
``K= K+1", \ \ ``\lambda=\frac{1}{K+1}\, \sqrt{\frac{\log_2(4/\eta)}{2}}" \ \text{ and } \ ``\alpha=\frac{p}{\sqrt{2\ln(4/\eta)}}",
\end{equation}
and observing that
\begin{equation} \label{e4.33}
\sqrt{2/p}\, (12+\lambda) (K+1) \mik (12+\lambda)(K+1)\alpha^{-1} \mik C(K,p,\eta),
\end{equation}
by \eqref{e4.32} and the choice of $C(K,p,\eta)$ in \eqref{e4.1}. Indeed, clearly we may assume that
$\|\boldsymbol{\theta}\|_2=1$. Therefore, if\, $\|\boldsymbol{\theta}\|_\infty\mik \alpha$, then,
by \eqref{e4.23} and the previous observation,
\begin{equation} \label{e4.34}
\mu_p \big(\{ H: \boldsymbol{X} \text{ is $C(K,p,\eta)$-subgaussian at the direction } \boldsymbol{\theta}_H\}\big) \meg 1- \eta,
\end{equation}
while if\, $\|\boldsymbol{\theta}\|_\infty\meg \alpha$, then, by \eqref{e4.24},
\begin{equation} \label{e4.35}
\mu_p \big(\{ H: \boldsymbol{X} \text{ is $C(K,p,\eta)$-subgaussian at the direction } \boldsymbol{\theta}_H\}\big) \meg p-\eta.
\end{equation}
\begin{rem} \label{r4.12}
Note that the lower bound in \eqref{e4.2} can be proved without invoking Proposition \ref{p4.8}.
Indeed, one can proceed using Corollary \ref{c4.5}, the elementary identity
\begin{equation} \label{e4.36}
\underset{H\sim \mu_p}{\ave} \|\boldsymbol{\theta}_H\|_2^2 = p \,  \|\boldsymbol{\theta}\|^2_2
\end{equation}
and Markov's inequality. However, this approach yields a weaker estimate for the constant $C(K,p,\eta)$
in \eqref{e4.1} and, more importantly, it provides no information on the behavior of the probability appearing
on the left-hand side of \eqref{e4.2}.
\end{rem}


\section{Comments}

\numberwithin{equation}{section} \label{sec5}

\subsection{Extension to non-linear functions} \label{subsec5.1}

Beyond the class of linear functions, Theorem \ref{t1.1} can be extended
to certain chaoses which have a natural combinatorial interpretation: they are the homomorphism densities associated with
weighted uniform hypergraphs (see, \emph{e.g.}, \cite[Chapter 7]{L}). Of course, in order to be meaningful such an extension,
one has to select an appropriate normalization. We will adopt the scaling which appears in the bounded differences
inequality\footnote{This choice is not optimal for certain classes of functions, but it appears to be the right choice
at this level of generality.}.

\subsubsection{ \ } \label{subsubsec5.1.1}

Specifically, let $n$ be a positive integer, and let $f\colon [-1,1]^n\to\rr$ be a bounded measurable function.
For every $i\in [n]$ set
\begin{equation} \label{e5.1}
\begin{split}
\Delta_i(f)\coloneqq \sup\big\{ |f(x_1,\dots,x_n)& - f(x_1,\dots,x_{i-1},x'_i,x_{i+1},\dots, x_n)| : \\
& x_1,\dots,x_n,x'_i\in [-1,1]\big\},
\end{split}
\end{equation}
and define
\begin{equation} \label{e5.2}
\|f\|_{\Delta}\coloneqq \big( \Delta_1(f)^2+\dots+\Delta_n(f)^2\big)^{1/2}.
\end{equation}
Notice that: (i)  the quantity $\|\cdot \|_{\Delta}$ is a semi-norm, (ii) $\|f+c\|_{\Delta}=\|f\|_{\Delta}$ for every $c\in \rr$,
(iii) $\|f\|_{\Delta}=0$ if and only if the function $f$ is constant, and (iv) if $f$ is linear, that is,
$f(x_1,\dots,x_n)=\theta_1 x_1+\dots+\theta_n x_n$, then $\|f\|_{\Delta}=2\|(\theta_1,\dots,\theta_n)\|_2$.

\subsubsection{ \ }  \label{subsubsec5.1.2}

Next, let $\boldsymbol{X}$ be a random vector in $\rr^n$ with $[-1,1]$-valued entries.
Given $K>0$, we say that $\boldsymbol{X}$ \textit{is $K$-subgaussian with respect to $f$} if
\begin{equation} \label{e5.3}
\|f(\boldsymbol{X})\|_{\psi_2} \mik K\, \|f\|_{\Delta}.
\end{equation}
Observe that if $f(x_1,\dots,x_n)=\theta_1 x_1+\dots+\theta_n x_n$ is linear, then this is equivalent to saying that
$\boldsymbol{X}$ is $K$-subgaussian at the direction $(\theta_1,\dots,\theta_n)$. Also note that if the random vector
$\boldsymbol{X}$ has independent entries, then the bounded differences inequality yields that $\boldsymbol{X}$ is
$O(1)\text{-subgaussian}$ with respect to $f-\ave[f(\boldsymbol{X})]$.

\subsubsection{ \ } \label{subsubsec5.1.3}

It is also straightforward to extend \eqref{e1.4}. Precisely, for every subset $H$ of~$[n]$ let
$f_H\colon [-1,1]^n\to\rr$ denote the function defined by
\begin{equation} \label{e5.4}
f_H(x_1,\dots,x_n)\coloneqq f\big( \pi_H(x_1,\dots,x_n)\big)
\end{equation}
where $\pi_H(x_1,\dots,x_n)=(x'_1,\dots,x'_n)$ with $x'_i=x_i$ if $i\in H$, and $x'_i=0$ otherwise.

Thus, the non-linear version of the question discussed in the introduction is whether the subgaussian behavior of the random
vector $\boldsymbol{X}$ with respect to the function $f$ is reflected to/characterized by the typical subgaussian behavior
of $\boldsymbol{X}$ with respect to $f_H$ where $H$ is random subset of $[n]$.

\subsubsection{ \ } \label{subsubsec5.1.4}

It is likely that this problem is rather delicate. As we have mentioned, we will consider the case
where the function $f$ is the homomorphism density associated with a weighted uniform hypergraph.

More precisely, let $d$ be a positive integer. For every integer $n\meg d$ and every $A\subseteq [n]$
by ${A\choose d}$ we denote the set of all subsets of $A$ of cardinality $d$. Let $\mathcal{W}$ be a weighted $d\text{-uniform}$ hypergraph,
that is, $\mathcal{W}$ is a map which assigns to every hyperedge $e\in {[n]\choose d}$ a weight $\mathcal{W}(e)\in \rr$.
The homomorphism density function associated with $\mathcal{W}$ is the map $\hom_{\mathcal{W}}\colon [-1,1]^n\to\rr$ defined by
\begin{equation} \label{e5.5}
\hom_{\mathcal{W}}(x_1,\dots, x_n)\coloneqq \sum_{e\in{[n]\choose d}} \mathcal{W}(e) \prod_{i\in e} x_i.
\end{equation}
Note that if $H$ is a subset of $[n]$, then the restriction $(\hom_{\mathcal{W}})_H$ of $\hom_{\mathcal{W}}$
defined in~\eqref{e5.4} is naturally identified with the homomorphism density function $\hom_{\mathcal{W}[H]}$
associated with the induced on $H$ sub-hypergraph $\mathcal{W}[H]$ of $\mathcal{W}$.

\subsubsection{ \ } \label{subsubsec5.1.5}

We have the following theorem.
\begin{thm} \label{t5.1}
The following hold.
\begin{enumerate}
\item[(1)] Let $K>0$, let\, $0<p<1$, and let $d$ be a positive integer. Also let $n\meg d$ be an integer,
let $\boldsymbol{X}$ be a random vector in $\rr^n$ with $[-1,1]\text{-valued}$ entries, and let $\mathcal{W}$
be a weighted $d$-uniform hypergraph on~$[n]$. If\, $\boldsymbol{X}$ is $K\text{-subgaussian}$ with respect
to $\hom_{\mathcal{W}}$, then for every $C>0$
\begin{equation} \label{e5.6}
\mu_p\big( \{ H: \boldsymbol{X} \text{ is $C$-subgaussian with respect to } \hom_{\mathcal{W}[H]}\}\big)
 \meg p^d-o_{C\to\infty;K,p,d}(1).
\end{equation}
\item[(2)] Conversely, let $K>0$, let\, $0<p<1$, let $0<\gamma\mik 1$, and let $d$ be a positive integer.
Also let $n\meg d$ be an integer, let $\boldsymbol{X}$ be a random vector in\, $\rr^n$ with $[-1,1]\text{-valued}$ entries,
and let $\mathcal{W}$ be a weighted $d\text{-uniform}$ hypergraph on~$[n]$. If
\[ \mu_p\big( \{ H: \boldsymbol{X} \text{ is $K$-subgaussian with respect to } \hom_{\mathcal{W}[H]}\}\big)\meg \gamma, \]
then $\boldsymbol{X}$ is $O_{K,p,\gamma,d}(1)$-subgaussian with respect to $\hom_{\mathcal{W}}$.
\end{enumerate}
\end{thm}
The proof of Theorem \ref{t5.1} is similar to the proof of Theorem \ref{t1.1}; for the convenience of the reader
we present the details in the Appendix.

We also note that the lower bound in \eqref{e5.6} is optimal. Specifically, we have the following analogue of Example \ref{ex4.2}.
\begin{examp} \label{ex5.1}
Fix a positive integer $d$, and let $n\meg d$ be an arbitrary integer. We define a weighted
$d$-uniform hypergraph $\mathcal{E}$ on $[n+d]$ by the rule
\[ \mathcal{E}(e)\coloneqq
\begin{cases}
{[n]\choose d} & \text{if } e=\{1,\dots,d\}, \\
-1 & \text{if } e\subseteq \{d+1,\dots, n+d\}, \\
0 & \text{otherwise}.
\end{cases} \]
Also fix a $[-1,1]$-valued random variable $Z$, and let $\boldsymbol{X}=(X_1,\dots,X_{n+d})$ be the random vector in $\rr^{n+d}$
defined by setting $X_i=Z$ for every $i\in [n+d]$. Observe that $\hom_{\mathcal{E}}(\boldsymbol{X})=0$, and so $\boldsymbol{X}$
is $K\text{-subgaussian}$ with respect to $\hom_{\mathcal{E}}$ for any~$K>0$. Next, let $0<p<1$ be arbitrary, and set
\[ \mathcal{H}\coloneqq \big\{H\subseteq [n+d]:\{1,\dots,d\}\nsubseteq H \text{ and } |H\cap\{d+1,\dots,n+d\}|\meg pn/2\big\}. \]
By~Proposition \ref{p2.2}, we see that $\mu_p(\mathcal{H})= 1-p^d-o_{n\to\infty;p,d}(1)$. Fix $H\in\mathcal{H}$
and set $G\coloneqq H\cap \{d+1,\dots, n+d\}$. Since $\hom_{\mathcal{E}[H]}(\boldsymbol{X})= -{G\choose d} Z^d$, we have
\[ \|Z^d\|_{\psi_2} \big(p^d/2^d-o_{n\to\infty;p,d}(1)\big)\, n^d \mik \|\hom_{\mathcal{E}[H]}(\boldsymbol{X})\|_{\psi_2}. \]
On the other hand, note that
\begin{enumerate}
\item[$\bullet$] $\Delta_i(\hom_{\mathcal{E}[H]})=0$ if $i\in \{1,\dots,d\}$ (this is because $\{1,\dots,d\}\nsubseteq H$),
\item[$\bullet$] $\Delta_i(\hom_{\mathcal{E}[H]})=0$ if $i\in \{d+1,\dots,n+d\}\setminus H$, and
\item[$\bullet$] $|\Delta_i(\hom_{\mathcal{E}[H]})| \mik 2 {G\choose d-1} \mik 2n^{d-1}$ if $i\in H\cap \{d+1,\dots,n+d\}$
\end{enumerate}
which implies that $\|\hom_{\mathcal{E}[H]}\|_{\Delta}\mik 2 n^{d-1/2}$. Therefore, if $K$ is any positive real such that
$\boldsymbol{X}$ is $K\text{-subgaussian}$ with respect to $\hom_{\mathcal{E}[H]}$, then
\[ K\meg \|Z^d\|_{\psi_2} \big(p^d/2^{d+1}-o_{n\to\infty;p,d}(1)\big) \sqrt{n}. \]
Thus, for any $C>0$ we have
\[  \mu_p\big( \{ H: \boldsymbol{X} \text{ is $C$-subgaussian with respect to } \hom_{\mathcal{E}[H]}\}\big)\mik p^d+o_{n\to\infty;p,d,C}(1). \]
\end{examp}

\subsection{Extension to partially subgaussian random vectors} \label{subsec5.2}

Let $n$ be a positive integer, and let $\boldsymbol{X}$ be a random
vector in $\rr^n$. Given $K,\tau>0$ and $\boldsymbol{\theta}\in \rr^n$, we say\footnote{This terminology is not
standard.} that $\boldsymbol{X}$ is \emph{$(K,\tau)$-partially subgaussian at the direction} $\boldsymbol{\theta}$ provided that
\begin{equation} \label{e5.7}
\mathbb{P}\big( \big\{|\langle \boldsymbol{\theta},\boldsymbol{X}\rangle|\meg t\big\}\big) \mik
2\exp\Big( -\frac{t^2}{K^2\|\boldsymbol{\theta}\|^2_2}\Big) \ \ \ \text{for every } t\meg \tau.
\end{equation}
Notice that if $\tau=O(\|\boldsymbol{\theta}\|_2)$, then this is equivalent to saying that
the random vector $\boldsymbol{X}$ is $O_{K}(1)$-subgaussian at the direction $\boldsymbol{\theta}$.
Thus, this notion is of interest when~$\tau$~is significantly larger than $\|\boldsymbol{\theta}\|_2$.
Examples of random vectors which are partially subgaussian with parameters in this regime appear frequently
in combinatorics, most~notably in various density increment strategies. Specifically, one encounters random
vectors in $\rr^n$ which are $(K,\tau)$-partially subgaussian at the direction $(1,\dots,1)\in \rr^n$ with
$K=O(1)$ and $\tau=\eta n$ where $\eta>0$ is a very small constant; see \cite[Part 2]{DK}. The understanding
of the statistical/concentration properties of these examples was the starting point of the present paper.

\subsubsection{ \ } \label{subsubsec5.2.1}

It is not hard to see that Theorem \ref{t1.1} can be extended to $(K,\tau)$-partially subgaussian
random vectors, but of course one is also interested in determining the quantitative dependence on the parameter $\tau$.
In this direction we have the following analogue of Proposition \ref{p4.8}.
\begin{prop} \label{p5.3}
Let $K\meg 1/\sqrt{2}$, let\, $0<p<1$, and let $\tau\meg\max\{p^{-1},\sqrt{2}K\}$.~Also let $n$ be a positive integer,
let $\boldsymbol{X}$ be a random vector in $\rr^n$ with $[-1,1]\text{-valued}$ entries, and let $\boldsymbol{\theta}\in \rr^n$
with $\|\boldsymbol{\theta}\|_2=1$. Assume that $\boldsymbol{X}$ is $(K,\tau)\text{-partially}$ subgaussian at the
direction~$\boldsymbol{\theta}$. Finally, let $0<\alpha\mik 1$. Then the following hold.
\begin{enumerate}
\item [(i)] If\, $\|\boldsymbol{\theta}\|_\infty\mik \alpha$, then
\begin{equation} \label{e5.8}
\begin{split}
\mu_p \Big(\big\{ H: \boldsymbol{X} & \text{ is $(2K/p,2p\tau)$-partially subgaussian at the direction } \boldsymbol{\theta}_H\big\}\Big) \\
& \meg 1- 3\exp\Big(- \frac{p^2\tau^2}{2K^2} \Big) -  2\exp\Big(-\frac{p^2}{2\alpha^2}\Big).
\end{split}
\end{equation}
\item [(ii)] If\, $\|\boldsymbol{\theta}\|_\infty\meg \alpha$, then
\begin{equation} \label{e5.9}
\begin{split}
\mu_p \Big(\big\{ H: \boldsymbol{X} & \text{ is $(2\sqrt{2}K/\alpha,2p\tau)$-partially subgaussian at the direction }
\boldsymbol{\theta}_H\big\}\Big) \\
& \meg p - 3\exp\Big(- \frac{p^2\tau^2}{2K^2} \Big).
\end{split}
\end{equation}
\end{enumerate}
\end{prop}
In particular, Proposition \ref{p5.3} yields that if $K\meg 1/\sqrt{2}$, $\tau=\eta n$ for some $\eta>0$,
$n\meg \max\{2K^2,p^{-2}\}/\eta^2$, and the random vector $\boldsymbol{X}$ is $(K,\tau)\text{-partially}$ subgaussian
at the direction $(1,\dots,1)\in \rr^n$, then the probability on the left-hand side of \eqref{e5.8} is at least
\begin{equation} \label{e5.10}
1- 3\exp\Big(- \frac{p^2\eta^2n}{2K^2} \Big) -  2\exp\Big(-\frac{p^2n}{2}\Big);
\end{equation}
that is, we have an exponential improvement upon \eqref{e4.31}.

\subsubsection{ \ } \label{subsubsec5.2.2}

Not surprisingly, the proof of Proposition \ref{p5.3} follows the lines of the proof of Proposition~\ref{p4.8}.
The only difference is that, instead of Corollary \ref{c4.5}, it uses a straightforward variant of Lemma \ref{l4.6} for partially
subgaussian random vectors. (In particular, the exponential gain in \eqref{e5.10} comes from the fact that we need to control
the tails up to~$\tau$.) We leave the details to the interested reader.

\subsection{Extension to not necessarily bounded random vectors} \label{5.3}

It is open to us whether part (1) of Theorem \ref{t1.1}
can be extended to random vectors with subgaussian, but not necessarily bounded, entries. Although the boundedness of
$\boldsymbol{X}$ is used only in \eqref{e4.13}, the strategy of our proof uses this property in an essential way
and it cannot be dropped by merely optimizing the argument.


\appendix

\section{Proof of Theorem \ref{t5.1}} \label{secA}

\numberwithin{equation}{section}

\subsection{Preliminary tools} \label{subsecA.1}

We begin by observing the following two simple facts; they will be used
in the proofs of both parts of Theorem \ref{t5.1}.
\begin{fact} \label{fa.1}
Let $0<p<1$, let $d,n$ be positive integers with $d\mik n$, and let $\mathcal{W}$ be a weighted $d$-uniform hypergraph on $[n]$.
Then, for any $\boldsymbol{x}\in [-1,1]^n$ we have
\begin{equation} \label{ea.1}
p^d\, \hom_{\mathcal{W}}(\boldsymbol{x})=\underset{H\sim \mu_p}{\ave} \hom_{\mathcal{W}[H]}(\boldsymbol{x}).
\end{equation}
In particular, if $\boldsymbol{X}$ is a random vector in $\rr^n$ with $[-1,1]$-valued entries, then
\begin{equation} \label{ea.2}
p^d\, \|\hom_{\mathcal{W}}(\boldsymbol{X})\|_{\psi_2} \mik \underset{H\sim \mu_p}{\ave} \|\hom_{\mathcal{W}[H]}(\boldsymbol{X})\|_{\psi_2}.
\end{equation}
\end{fact}
\begin{proof}
Write $\boldsymbol{x}=(x_1,\dots,x_n)$ and notice that
\begin{eqnarray*}
\underset{H\sim \mu_p}{\ave} \hom_{\mathcal{W}[H]}(\boldsymbol{x}) & = &
\underset{H\sim \mu_p}{\ave} \sum_{e\in {H \choose d}} \mathcal{W}(e) \prod_{i\in e} x_i
 = \underset{H\sim \mu_p}{\ave} \sum_{e\in {[n] \choose d}} \mathcal{W}(e) \prod_{i\in e} x_i \mathbf{1}_H(i) \\
& = & \sum_{e\in {[n] \choose d}} \mathcal{W}(e)\underset{H\sim \mu_p}{\ave} \prod_{i\in e} x_i \mathbf{1}_H(i)
 = \sum_{e\in {[n] \choose d}} \mathcal{W}(e)\, p^d\, \prod_{i\in e} x_i \\
& = & p^d\, \hom_{\mathcal{W}}(\boldsymbol{x}).
\end{eqnarray*}
The estimate in \eqref{ea.2} follows from \eqref{ea.1} and the triangle inequality.
\end{proof}
\begin{fact}\label{fa.2}
Let $n$ be a positive integer, let $\boldsymbol{X}$ be a random vector in $\rr^n$ with
$[-1,1]\text{-valued}$ entries, and let $f\colon [-1,1]^n\to\rr$ be a bounded measurable function. Define
$g\colon \{0,1\}^n\to\rr$ by setting $g(H)=\| f_H(\boldsymbol{X})\|_{\psi_2}$ for every $H\subseteq [n]$, where
$f_H$ is as in \eqref{e5.4}. Then we have
\begin{equation} \label{ea.3}
\|g\|_\Delta \mik \frac{\|f\|_\Delta}{\sqrt{\ln2}}.
\end{equation}
\end{fact}
\begin{proof}
The desired estimate is a consequence of the fact that for every bounded random variable $Y$ we have
\begin{equation} \label{ea.4}
\|Y\|_{\psi_2}\mik \frac{\|Y\|_{L_{\infty}}}{\sqrt{\ln2}}.
\end{equation}
Indeed, fix $i\in[n]$ and $H\subseteq [n]\setminus\{i\}$, and observe that
\begin{eqnarray*}
\Delta_i(f) & \stackrel{\eqref{e5.1}}{\meg} & \sup\{ |f_{H\cup\{i\}}(\boldsymbol{x}) - f_{H}(\boldsymbol{x})| : \boldsymbol{x}\in[-1,1]^n \} \\
& \stackrel{\eqref{ea.4}}{\meg} & \sqrt{\ln2}\; \|f_{H\cup\{i\}}(\boldsymbol{X}) - f_{H}(\boldsymbol{X})\|_{\psi_2}
\meg \sqrt{\ln2}\; | g(H\cup\{i\}) - g(H)|.
\end{eqnarray*}
Thus, we have $\Delta_i(g)\mik\Delta_i(f)/\sqrt{\ln2}$ for every $i\in [n]$. This, in turn, implies inequality \eqref{ea.3}.
\end{proof}

\subsection{Proof of part (2)} \label{subsecA.2}

Let $K, p, \gamma, d, n, \boldsymbol{X}, \mathcal{W}$ be as in part (2) of Theorem \ref{t5.1}, and~set
\begin{equation} \label{ea.5}
C=C(K,p,\gamma,d)\coloneqq \frac{1}{p^d}\, \big(K+\sqrt{1-\log_2(\gamma)}\big).
\end{equation}
We will show that if
\begin{equation} \label{ea.6}
\mu_p\big( \{ H: \boldsymbol{X} \text{ is $K$-subgaussian with respect to } \hom_{\mathcal{W}[H]}\}\big)\meg \gamma,
\end{equation}
then $\boldsymbol{X}$ is $C$-subgaussian with respect to $\hom_{\mathcal{W}}$. To this end we need the following lemma.
\begin{lem} \label{la.3}
Let $p,d,n,\boldsymbol{X},\mathcal{W}$ be as in part \emph{(2)} of Theorem \emph{\ref{t5.1}}. Then, setting
\begin{equation} \label{ea.7}
M\coloneqq \underset{H\sim \mu_p}{\ave} \|\hom_{\mathcal{W}[H]}(\boldsymbol{X})\|_{\psi_2},
\end{equation}
for any $t>0$ we have
\begin{equation} \label{ea.8}
\mu_p\big( \big\{H: \big| \|\hom_{\mathcal{W}[H]}(\boldsymbol{X})\|_{\psi_2}-M\big|\meg t\big\}\big)\mik
2\exp\Big(-\frac{2\ln 2\,t^2}{\|\hom_{\mathcal{W}}\|^2_{\Delta}}\Big).
\end{equation}
\end{lem}
\begin{proof}
Define $g\colon \{0,1\}^n\to\rr$ by setting $g(H)=\|\hom_{\mathcal{W}[H]}(\boldsymbol{X})\|_{\psi_2}$ for every $H\subseteq [n]$, and observe
that $\underset{H\sim \mu_p}{\ave} g(H)= M$. By Fact \ref{fa.2} applied for the function ``$f=\hom_{\mathcal{W}}$", we see that
$\|g\|_{\Delta}\mik \|\hom_{\mathcal{W}}\|_{\Delta}/ \sqrt{\ln 2}$. Hence, by Proposition~\ref{p2.3}, for any $t>0$ we have
\[ \begin{split}
\mu_p\big( \big\{H: \big| \|\hom_{\mathcal{W}[H]}(\boldsymbol{X})\|_{\psi_2} & - M\big|\meg t\big\}\big) =
\mu_p\big( \{H: |g(H)-M|\meg t\}\big) \\
& \mik 2\exp\Big(-\frac{2 t^2}{\|g\|^2_{\Delta}}\Big) \mik 2\exp\Big(-\frac{2\ln 2\,t^2}{\|\hom_{\mathcal{W}}\|^2_{\Delta}}\Big)
\end{split} \]
as desired.
\end{proof}
Now set $t_0\coloneqq \sqrt{1-\log_2(\gamma)}\, \|\hom_{\mathcal{W}}\|_{\Delta}$, and let $M$ be as in \eqref{ea.7}.
By \eqref{ea.6} and Lemma \ref{la.3}, there exists $H_0\subseteq [n]$ such that
\begin{enumerate}
\item[$\bullet$] $M\mik t_0+ \|\hom_{\mathcal{W}[H_0]}(\boldsymbol{X})\|_{\psi_2}$, and
\item[$\bullet$] $\|\hom_{\mathcal{W}[H_0]}(\boldsymbol{X})\|_{\psi_2}\mik K \|\hom_{\mathcal{W}[H_0]}\|_{\Delta}
\mik K\|\hom_{\mathcal{W}}\|_{\Delta}$.
\end{enumerate}
(The last inequality follows from the definition of the semi-norm $\|\cdot\|_{\Delta}$ and \eqref{e5.4}.)
Using these estimates, the result follows by \eqref{ea.2} and the choice of $C$ in \eqref{ea.5}.

\subsection{Proof of part (1)} \label{subsecA.3}

The proof of this part is more involved. As we have already noted,
the argument is similar to that of the proof of Theorem \ref{t4.1}.

\subsubsection{A large deviation inequality} \label{subsubsecA.3.1}

The first step is the following analogue of Proposition~\ref{p4.3}.
\begin{prop} \label{pa.4}
Let $K\meg 1/\sqrt{2}$, and let $0<p<1$. Also let $d,n, \boldsymbol{X},\mathcal{W}$ be as in part \emph{(1)}
of Theorem \emph{\ref{t5.1}}. If\, $\boldsymbol{X}$ is $K\text{-subgaussian}$ with respect to $\hom_\mathcal{W}$,
then for any $\lambda\meg 8\sqrt{2}$,
\begin{equation} \label{ea.9}
\mu_p\big( \big\{ H: \|\hom_{\mathcal{W}[H]}(\boldsymbol{X})\|_{\psi_2} \meg
\lambda K\, \|\hom_{\mathcal{W}}\|_\Delta \big\}\big) \mik 3\exp\Big(-\frac{\ln 2}{32}\lambda^2\Big).
\end{equation}
\end{prop}
\begin{proof}
Note that, arguing as in Subsection \ref{subsec4.3}, it is enough to show the following.
\medskip

\noindent \textit{Let $K>0$, and let $0<p<1$. Also let $d,n, \boldsymbol{X},\mathcal{W}$ be as in
part \emph{(1)} of Theorem \emph{\ref{t5.1}}. If $\boldsymbol{X}$ is $K\text{-subgaussian}$ with respect to $\hom_\mathcal{W}$,
then, setting $Q\coloneqq \max\{2p^dK,\sqrt{2}\}$, for every $M\meg \max\{4\sqrt{2\ln 2}\, p^dK , 4\sqrt{\ln 2}\}$ we have }
\begin{equation} \label{ea.10}
\mu_p\Big( \Big\{ H: \| \hom_{\mathcal{W}[H]}(\boldsymbol{X}) \|_{\psi_2} \mik
\sqrt{\frac{3}{\ln 2}} \,M\,\|\hom_{\mathcal{W}}\|_\Delta \Big\}\Big) \meg 1- 3 \exp\Big(-\frac{M^2}{2Q^2}\Big).
\end{equation}

\noindent The left-hand side of \eqref{ea.10} is scale-invariant, and so we may assume that the weighted hypergraph $\mathcal{W}$
satisfies $\|\hom_{\mathcal{W}}\|_\Delta=1$. Thus, it is enough to prove that
\begin{equation} \label{ea.11}
\mu_p\Big( \Big\{ H: \| \hom_{\mathcal{W}[H]}(\boldsymbol{X})  \|_{\psi_2} \mik
\sqrt{\frac{3}{\ln 2}}\,M \Big\}\Big) \meg 1- 3 \exp\Big(-\frac{M^2}{2Q^2}\Big)
\end{equation}
for every $M\meg \max\{4\sqrt{2\ln 2}\, p^dK , 4\sqrt{\ln 2}\}$.

As in Lemma \ref{l4.6}, we start by showing that for any $t>0$ we have
\begin{equation} \label{ea.12}
\mu_p\Big(\Big\{ H: \mathbb{P}\big(\big\{|\hom_{\mathcal{W}[H]}(\boldsymbol{X}) |\meg t\big\}\big)
\mik 2 \exp\Big(-\frac{t^2}{2Q^2}\Big)\Big\}\Big) \meg  1- 2 \exp\Big(-\frac{t^2}{2Q^2}\Big).
\end{equation}
Fix $t>0$ and let $(\Omega,\mathcal{F},\mathbb{P})$ denote the underlying probability space.
Let $\omega\in\Omega$ be arbitrary, and recall that $\boldsymbol{X}(\omega)\in [-1,1]^n$. We define
the map $\zeta\colon \{0,1\}^n\to \rr$ by setting $\zeta(H)=\hom_{\mathcal{W}[H]}\big(\boldsymbol{X}(\omega)\big)$
for every $H\subseteq [n]$; observe that $\Delta_i(\zeta)\mik \Delta_i(\hom_{\mathcal{W}})$ for every $i\in [n]$.
Since $\|\hom_{\mathcal{W}}\|_\Delta=1$, by Proposition \ref{p2.3} and identity \eqref{ea.1},
\begin{equation} \label{ea.13}
\mu_p \Big(\Big\{ H\! : \big| \hom_{\mathcal{W}[H]}\big(\boldsymbol{X}(\omega)\big) -
p^d \hom_{\mathcal{W}}\big(\boldsymbol{X}(\omega)\big) \big|< \frac{t}{2} \Big\}\Big)
\meg 1 - 2 \exp\Big(-\frac{t^2}{2}\Big).
\end{equation}
(Note that \eqref{ea.13} is the analogue of \eqref{e4.13}. We point out that this is, essentially, the only step
of the proof which differs from that of Proposition \ref{p4.3}.) Also observe that the event
\begin{equation} \label{ea.14}
\big\{ (H,\omega): \big|\hom_{\mathcal{W}[H]}\big(\boldsymbol{X}(\omega)\big)\big|<t\big\}
\end{equation}
contains the event
\begin{equation} \label{ea.15}
\begin{split}
\Big\{ (H,\omega):  \big|\hom_{\mathcal{W}}&\big(\boldsymbol{X}(\omega)\big)\big|<\frac{t}{2p^d} \ \text{ and }\\
& \big|\hom_{\mathcal{W}[H]}\big(\boldsymbol{X}(\omega)\big) -
p^d\,\hom_{\mathcal{W}}\big(\boldsymbol{X}(\omega)\big)\big|< \frac{t}{2}\Big\}.
\end{split}
\end{equation}
On the other hand, we have $\|\hom_{\mathcal{W}}(\boldsymbol{X})\|_{\psi_2}\mik K$ since $\|\hom_{\mathcal{W}}\|_\Delta=1$
and the random vector $\boldsymbol{X}$ is $K$-subgaussian with respect to $\hom_{\mathcal{W}}$.
Thus, by Proposition~\ref{p2.1},
\begin{equation} \label{ea.16}
\mathbb{P}\Big( \Big\{ |\hom_{\mathcal{W}}(\boldsymbol{X})|<\frac{t}{2p^d} \Big\}\Big)\meg
1-2\exp\Big(-\frac{t^2}{(2p^dK)^2}\Big).
\end{equation}
Denoting by $\mu_p\times\mathbb{P}$ the product probability measure of $\mu_p$ and $\mathbb{P}$,
the previous discussion yields that
\begin{equation} \label{ea.17}
\mu_p\times\mathbb{P} \big(\big\{ (H,\omega) : \big|\hom_{\mathcal{W}[H]}\big(\boldsymbol{X}(\omega)\big)\big|
\meg t\big\}\big) \mik 4\exp\Big(-\frac{t^2}{Q^2}\Big).
\end{equation}
The estimate in \eqref{ea.12} now follows from \eqref{ea.17} and Markov's inequality.

With inequality \eqref{ea.12} at our disposal, we will estimate the probability in \eqref{ea.11} using Sublemma \ref{s4.7}.
Precisely, fix $M\meg \max\{4\sqrt{2\ln 2}\,p^dK , 4\sqrt{\ln 2}\}$, and for every $j\in\nn$ set
\begin{equation} \label{ea.18}
\mathcal{C}_M^j\coloneqq \Big\{H: \mathbb{P}\big(\big\{ |\hom_{\mathcal{W}[H]}(\boldsymbol{X})|\meg 2^j M\big\}\big)
\mik 2\exp\Big(-\frac{2^{2j}M^2}{2Q^2}\Big)\Big\}
\end{equation}
Also set
\begin{equation} \label{ea.19}
\mathcal{C}_M \coloneqq \bigcap_{j\in\nn} \mathcal{C}_M^j.
\end{equation}
By \eqref{ea.12}, we have $\mu_p(\mathcal{C}_M^j)\meg 1- 2\exp\Big(-\frac{2^{2j}M^2}{2Q^2}\Big)$ for every $j\in\mathbb{N}$.
This estimate and the fact that $M\meg 2\sqrt{2\ln 2}\, Q \meg \sqrt{2\ln 2}\, Q$ are easily seen to imply that
\begin{equation} \label{ea.20}
\mu_p(\mathcal{C}_M) \meg 1- 3 \exp(-M^2/2Q^2).
\end{equation}
For every $H\in \mathcal{C}_M$, by Sublemma \ref{s4.7} applied for the random variable ``$X = \hom_{\mathcal{W}[H]}(\boldsymbol{X})$"\!,
``$R=M$" and ``$C=\sqrt{2}\,Q$" and using again the fact that $M\meg 2\sqrt{2\ln 2}\, Q$, we obtain that
$\|\hom_{\mathcal{W}[H]}(\boldsymbol{X})\|_{\psi_2}\mik \sqrt{3/\ln 2}\,M$. That is, \eqref{ea.11} is satisfied, as desired.
\end{proof}

\subsubsection{Consequences} \label{subsubsecA.3.2}

We will need two consequences of Proposition \ref{pa.4}. The first one is the analogue of Corollary \ref{c4.4};
its proof is identical to that of Corollary \ref{c4.4}.
\begin{cor} \label{ca.5}
Let $K\meg 1/\sqrt{2}$, and let $0<p<1$. Also let $d,n, \boldsymbol{X},\mathcal{W}$ be as in part \emph{(1)} of
Theorem \emph{\ref{t5.1}}. If\, $\boldsymbol{X}$ is $K\text{-subgaussian}$ with respect to $\hom_\mathcal{W}$, then
\begin{equation} \label{ea.21}
\underset{H\sim \mu_p}{\ave} \|\hom_{\mathcal{W}[H]}(\boldsymbol{X})\|_{\psi_2}\mik 12 K\,\|\hom_{\mathcal{W}}\|_\Delta.
\end{equation}
\end{cor}
The second corollary is the analogue of Corollary \ref{c4.5}.
\begin{cor} \label{ca.6}
Let $K\meg 1/\sqrt{2}$, and let $0<p<1$. Also let $d,n, \boldsymbol{X},\mathcal{W}$ be as in part \emph{(1)} of
Theorem \emph{\ref{t5.1}}. If\, $\boldsymbol{X}$ is $K\text{-subgaussian}$ with respect to $\hom_\mathcal{W}$, then for any $\lambda> 0$,
\begin{equation} \label{ea.22}
\mu_p\big(\big\{H: \|\hom_{\mathcal{W}[H]}(\boldsymbol{X})\|_{\psi_2} \meg (12+\lambda)K\,\|\hom_{\mathcal{W}}\|_\Delta \big\}\big)
\mik 2\exp(-2\ln 2\,\lambda^2K^2).
\end{equation}
\end{cor}
\begin{proof}
As in the proof of Lemma \ref{la.3}, define the function $g\colon \{0,1\}^n \to \rr$ by setting
$g(H)= \|\hom_{\mathcal{W}[H]}(\boldsymbol{X})\|_{\psi_2}$ for every $H\subseteq[n]$. Recall that, by Fact \ref{fa.2}, we have
\begin{equation}\label{ea.23}
\|g\|_\Delta \mik \frac{\|\hom_{\mathcal{W}}\|_\Delta }{\sqrt{\ln2}}.
\end{equation}
Using Corollary \ref{ca.5} and \eqref{ea.23}, the result follows by applying Proposition \ref{p2.3} to the function $g$
and the vector ``$\mathbf{c}=\big(\Delta_1(g),\dots,\Delta_n(g)\big)$".
\end{proof}

\subsubsection{Completion of the proof} \label{subsubsecA.3.3}

Notice that part (1) of Theorem \ref{t5.1} follows from the following, more informative, theorem.
\begin{thm} \label{ta.7}
Let $K,p,d,n,\boldsymbol{X},\mathcal{W}$ be as in part \emph{(1)} of Theorem \emph{\ref{t5.1}}.
Also let $0<\eta<p^d$, and set
\begin{equation}\label{ea.24}
C= C(K,d,\eta) \coloneqq \frac{26}{\eta}\, (K+1)\, \sqrt{\frac{1}{2}\, \log_2\Big(\frac{4}{\eta}\Big)}.
\end{equation}
If $\boldsymbol{X}$ is $K$-subgaussian with respect to $\hom_\mathcal{W}$, then
\begin{equation} \label{ea.25}
\mu_p\big( \{ H: \boldsymbol{X} \text{ is $C$-sub}\text{gaussian with respect to } \hom_{\mathcal{W}[H]}\}\big) \meg p^d-\eta.
\end{equation}
\end{thm}
\begin{proof}
Set
\begin{equation} \label{ea.26}
\lambda\coloneqq \frac{1}{K+1}\, \sqrt{\frac{1}{2}\log_2\Big(\frac{4}{\eta}\Big)}
\end{equation}
and observe that $2\exp(-2\ln2 \lambda^2(K+1)^2)=\eta/2$. Also set
\begin{equation}\label{ea.27}
\mathcal{H}_1\coloneqq \big\{H: \|\hom_{\mathcal{W}[H]}(\boldsymbol{X})\|_{\psi_2}\mik
(12+\lambda)(K+1)\, \|\hom_{\mathcal{W}}\|_\Delta \big\}
\end{equation}
and
\begin{equation} \label{ea.28}
\mathcal{H}_2\coloneqq \Big\{H: \|\hom_{\mathcal{W}[H]}\|_\Delta \meg \frac{\eta}{2}\,\|\hom_{\mathcal{W}}\|_\Delta\Big\}.
\end{equation}
By Corollary \ref{ca.6}, we have
\begin{equation} \label{ea.29}
\mu_p(\mathcal{H}_1)\meg 1-2\exp\big(-2\ln2\, \lambda^2(K+1)^2\big) = 1- \frac{\eta}{2}.
\end{equation}
On the other hand, by identity \eqref{ea.1}, the fact that $\|\cdot\|_{\Delta}$ is a semi-norm, and the triangle inequality,
we have $p^d\,\|\hom_{\mathcal{W}}\|_\Delta\mik \underset{H\sim \mu_p}{\ave} \|\hom_{\mathcal{W}[H]}\|_{\Delta}$.
Moreover, notice that $\|\hom_{\mathcal{W}[H]}\|_{\Delta} \mik \|\hom_{\mathcal{W}}\|_{\Delta}$ for every $H\subseteq [n]$.
Using these observations, we obtain that
\begin{equation} \label{ea.30}
\mu_p(\mathcal{H}_2) \meg p^d - \frac{\eta}{2}.
\end{equation}
Therefore, by \eqref{ea.29} and \eqref{ea.30}, we see that
\begin{equation} \label{ea.31}
\mu_p(\mathcal{H}_1\cap\mathcal{H}_2)\meg p^d -\eta.
\end{equation}
Finally observe that, by the choice of $C$ in \eqref{ea.24}, for every $H\in \mathcal{H}_1\cap\mathcal{H}_2$
we have $\|\hom_{\mathcal{W}[H]}(\boldsymbol{X})\|_{\psi_2}\mik C\|\hom_{\mathcal{W}[H]}\|_\Delta$.
The proof is completed.
\end{proof}
\begin{rem} \label{ra.9}
We note that it is also possible to obtain a partial extension of part~(i) of Proposition \ref{p4.8}. More precisely,
if $\mathcal{W}$ is the complete $d$-uniform hypergraph on $n$ vertices---that is, if $\mathcal{W}(e)=1$
for every $e\in {[n]\choose d}$---or, more generally, if the weighted hypergraph $\mathcal{W}$ is sufficiently
pseudorandom\footnote{The notion of pseudorandomness which is needed in our setting is the following requirement:
for every $0<p<1$, if $n$ is sufficiently large (depending only on $p$), then for every $i\in [n]$ and every
$H\subseteq [n]$ with $|H|\meg pn$ we have
\[ \sum_{i\in e\in {H\choose d}}\!\!\!\mathcal{W}(e) \meg
(p^{d-1}/2)\!\!\sum_{i\in e\in {[n]\choose d}}\!\!\!\mathcal{W}(e). \]}, then the probability on the left-hand side
of \eqref{e5.6} is $1-o_{C\to\infty;K,p,d}(1)-o_{n\to\infty;p,d}(1)$.
\end{rem}


\end{document}